\documentclass[11pt,twoside]{article}
\usepackage{latexsym}
\usepackage{amssymb,amsbsy,amsmath,amsfonts,amssymb,amscd}
\usepackage{amsthm}
\usepackage{epsfig, graphicx,subfigure}
\usepackage[active]{srcltx}
\usepackage{psfrag}
\usepackage{url}

\usepackage{colortbl}
\newcommand{\commentout}[1]{}
\newcommand{\R}{\mathbb{R}}

\newcommand {\Chi} {{\bf \raise 2pt \hbox{$\chi$}} }
\newcommand {\f}   {\frac}
\newcommand {\p}   {\partial}

\newcommand{\dis}{\displaystyle}
\newcommand{\beq}{\begin{equation}}
\newcommand{\beqa} {\begin{array}{rl}}
\newcommand{\eeq}{\end{equation}}
\newcommand{\eeqa}{\end{array}}
\newtheorem{theorem}{Theorem}

\newtheorem{remark}[theorem]{Remark}

\newtheorem{claim}[theorem]{Claim}

\title{\LARGE The contribution of age structure to cell population responses to targeted therapeutics}

\author{Pierre Gabriel\thanks{Universit\'e Pierre et Marie Curie-Paris 6, UMR 7598 LJLL, BC187, 4, place de Jussieu, F-75252 Paris cedex 5, France. Email: gabriel@ann.jussieu.fr}
\thanks{Present address: INRIA Rh\^ones-Alpes, projet BEAGLE, B\^atiment CEI-1, BP 52132, 66 Boulevard Niels Bohr, F-69603 Villeurbanne cedex, France.}\and
Shawn P. Garbett\thanks{Department of Cancer Biology, Vanderbilt University, 2220 Pierce Ave., Nashville, TN 37232}\and
Darren R. Tyson\footnotemark[2]\and
Glenn F. Webb\thanks{Department of Mathematics, Vanderbilt University, 1326 Stevenson Center, Nashville, TN 37240. Email: glenn.f.webb@vanderbilt.edu}}

\date{\today}

\begin{document}
\maketitle
\pagestyle{plain}
\pagenumbering{arabic}

\begin{abstract}
Cells grown in culture act as a model system for analyzing the effects of anticancer compounds, which may affect cell behavior in a cell cycle position-dependent manner.
Cell synchronization techniques have been generally employed to minimize the variation in cell cycle position.
However, synchronization techniques are cumbersome and imprecise and the agents used to synchronize the cells potentially have other unknown effects on the cells.
An alternative approach is to determine the age structure in the population and account for the cell cycle positional effects post hoc.
Here we provide a formalism to use quantifiable lifespans from live cell microscopy experiments to parameterize an age-structured model of cell population response.
\end{abstract}


\


\

\section{Introduction}

When examined individually in time lapse microscopy experiments, cells grown in culture display variability in the length of their cell cycles (between mitotic events),
and this variability is represented by intermitotic time (IMT) distributions \cite{Dowling,Lee,Liu}.
These distributions are usually obtained from asynchronously dividing populations of cells, which achieve a steady-state age structure when the population is growing exponentially.
In experimental studies that examine the effects of perturbations on cellular proliferation,
it is often desirable to determine whether the perturbation is affecting cells in particular stages within their cell cycle,
i.e. whether the perturbation has cell cycle-specific effects.
Since we do not know where a cell is in the cycle,
an alternative is to define the age of a cell as the time elapsed from its last division and deal with this measurable quantity instead of the cell cycle phase.
In this paper we adopt this definition and provide a formalism to convert data obtained as IMT distributions to parameterize an age-structured population model
and, thus, identifying the contribution of age structure to the response of the cell population to perturbation.

Models of age-structured populations using partial differential equations, such as those originally discribed by A. Lotka, A.G. McKendrick, W.O. Kermack, and F. von Foerster,
are well adapted to model the dynamical features of experimental cultures of cells transiting the cell cycle with variable IMT. 
These models have been widely studied from a mathematical perspective \cite{Ainseba,Ainseba2,Calsina,Clement,DHT,EngelNagel,MetzDiekmann,BP,Picart2,Webb},
but the application of the model to experimental data has been hampered by an inability to determine the age-dependent model parameters.
The usual approach for parameter estimation is to solve numerically an inverse problem
(see~\cite{Ackleh,Banks,Banks2,DPZ,DMZ,Engl,Gyllenberg,Picart,PilantRundell,PilantRundell2,Rundell,Rundell2,PZ} on this question for structured population models),
but this requires extensive input data and is specific to a given situation.
A much more convenient approach is to assume that the distributed parameters lie in a class of functions with only a few constants (power laws or other forms~\cite{PG})
and obtain the parameters from the fit of the functions to experimental data.

One function that often provides a better fit to IMT data than others, such as log-normal, inverse normal or gamma functions, is an exponentially modified Gaussian (EMG)~\cite{Golubev,TysonGarbett}.
Under conditions in which the IMT distribution can be explained by an EMG model, we submit that the age-dependent division rate can be identified as an error function.
Starting from this observation, we present a simple method to recover the parameters of this error function from parameters fitted to the experimental IMT data.
Once reliably parameterized, the age-structured model can be used to make predictions about cell age-dependent effects of perturbations, for example,
whether cells arrest during their cell cycle in response to treatment with antiproliferative compounds.

Individual based models (IBM) have also been used to track individual cell behavior in proliferating cell populations
\cite{Alarcon,Mallet,Gerlee,Patel,Powathil,Quaranta08,Quaranta05,RamisConde,Ribba}.
IBM have advantages in their direct connection to discrete events and to cell-cell interactions.
IBM are usually readily implementable to computer simulations, but may require many repeated runs to access an average outcome.
IBM are usually difficult to analyze theoretically with respect to parametric input, particularly with highly sensitive parameters.
The advantage of continuum differential equations models, such as the cell age structured models we develop here, are their tractability for theoretical analysis
and their reproducibility for specific parametric input. The practical difficulties of determining this parametric input is the problem we focus upon here.
IBM and differential equations models should be viewed as companion approaches,
which complement and compare their insights and predictions for proliferating cell population behavior, when individual cell variation is a primary consideration.

\

\section{The age-structured model}

As with human populations, we can associate an age to each individual cell in a cell population.
We define the age of a cell as the time elapsed from its last mitosis.
Starting from this definition, we derive an age-structured model giving the dynamics of populations of proliferating and quiescent cells.
This model is adapted to investigate the effect of treatment by the drug erlotinib on \textit{in vitro} PC-9 cancer cell lines.
It has been recently shown in~\cite{TysonGarbett} that the main effect of erlotinib on cancer cells is to induce entry into quiescence.
In~\cite{TysonGarbett} a system of ordinary differential equations  (without age structure) is used to model these experiments.
We hypothesize here that erlotinib induced quiescence is linked to the age of the cells involved.
Many age-structured models with quiescence can be found in the literature (see~\cite{Arino,Bekkal1,Bekkal2,GyllWebb1,GyllWebb2,GyllWebb3} for examples).
Here we present a simple age-structured model with quiescence which allows us to explain observed delays in response to erlotinib.
We start from the observation that there is no effect of the treatment on total population growth during the first twenty hours (see~\cite{TysonGarbett} and Figure~\ref{fig:compare}).
This time corresponds almost exactly to the minimal age of division observed for PC-9 cells.
It suggests that erlotinib acts only during a specific phase of the cell cycle, which based on its biological activity would be expected to be in G1.
The model we present is based on this idea and considers a fractional rate $f$ of cells that become quiescent in an age-dependent manner,
where the fraction is assumed to reflect the dose of erlotinib used for the treatment.

\

We start from the the McKendrick--Von Foerster's model which is widely used to model the cell cycle (see~\cite{MetzDiekmann,BP,Webb} and references).
The partial differential equation in this model provides the evolution of the density $p(t,a)$ of cells with age, or ``cell cycle phase'', $a$ at time~$t.$
To take into account the quiescent cells, we introduce the quantity $Q(t)$ which representsthe density of quiescent cells at time $t.$
It evolves according to an ordinary differential equation coupled to the equation on $p(t,a).$
We obtain the system
\beq\label{eq:quiescence}
\left\{\begin{array}{l}
\dis\f{\p}{\p t}p(t,a)+\f{\p}{\p a}p(t,a)+\beta(a)p(t,a)+\mu\, p(t,a)=0,\qquad t\geq0,\quad a>0,
\vspace{.3cm}\\
p(t,0)=\dis2(1-f)\int_0^\infty \beta(a)p(t,a)\,da,
\vspace{.3cm}\\
p(0,a)=p_0(a),
\vspace{.3cm}\\
\dis\f{d}{dt}Q(t)=2f\int_0^\infty \beta(a)p(t,a)\,da-\mu\, Q(t),
\vspace{.3cm}\\
Q(0)=Q_0.
\end{array}\right.
\eeq
In this model, the proliferaing cells age one-to-one with time at speed $\f{da}{dt}=1$, and divide with rate $\beta(a)\geq0.$
When a cell divides at mitosis, it produces two daughters which are either proliferating with age $a=0$, or quiescent.
This is taken into account by the boundary condition at $a=0$ and the first term in the equation for $Q(t),$
where the parameter $f\in[0,1]$ represents the fraction of proliferating cells which become quiescent.
The coefficient $\mu$ is a death rates.
The number of proliferating cells at time $t$ with age between $a_1$ and $a_2$ is 
$\int_{a_1}^{a_2}p(t,a) \, da$, and the total number of cells at time $t$ is $N(t)=\int_{0}^{\infty}p(t,a) \, da+Q(t).$

The division rate $\beta$ has a probabilistic interpretation: the probability that a cell did not divide by age $a$ is given by
$${\mathbb P}(a)=e^{-\int_0^a\beta(a')\,da'}.$$
Since all the proliferating cells must divide at some time by definition, the division rate has to satisfy
\beq\label{as:beta}\lim_{a\to+\infty}\int_0^a\beta(a')\,da'=+\infty.\eeq

The model is completed with intial data $p_0(a)$ and $Q_0.$
We choose the time $t=0$ to be the beginning of the erlotinib treatment, so at this time there are no quiescent cells $(Q_0=0)$.
The age distribution of the proliferating cells is assumed to be at equilibrium ({\it i.e.} $p_0(a)=const\,\hat p(a),$ see Section~\ref{sec:imt} for the mathematical definition of $\hat p$).
The experimental values of the total population $N(t)$ along time are ploted after normalization by the initial value $N(0)$ on a log-scale (see Figure~\ref{fig:compare}).
Because of this normalization, we consider an initial distribution such that $\int_0^\infty p_0(a)\,da=1$ which leads to $p_0(a)=\hat p(a)$ because of the definition of $\hat p.$

\

We want to compare the solutions of model~\eqref{eq:quiescence} to the experimental observations presented in~\cite{TysonGarbett}.
The first step is to estimate the different parameters of the model.
To estimate the value of the fraction $f$ for different doses of treatment, we use data available in~\cite{TysonGarbett}.
The fraction $F$ of quiescent cells is estimated by examination of whether cells treated with erlotinib divide or not before the end of the experiment.
We want to use this experimental fraction $F$ to estimate the coefficient $f$ of the model. 
For the sake of simplicity, assume that the death rate $\mu$ is $0$ in~\eqref{eq:quiescence}.
In this case, at the end of the labeling period $t_0$, the quantity of labeled quiescent cells corresponds to $Q(t_0)$ and the quantity of proliferating cells corresponds to $\int_0^{t_0}p(t,0)\,dt.$
Thus, the fraction $F$ of quiescent cells at $t_0$ is
\beq\label{eq:frac}F=\f{Q(t_0)}{Q(t_0)+\int_0^{t_0}p(t,0)\,dt}.\eeq
Now we compute this quantity from model~\eqref{eq:quiescence}, keeping in mind that we have assumed no mortality.
We have
\begin{align*}
Q(t_0)&=\int_0^{t_0}\f{dQ}{dt}(t)\,dt\\
&=2f\int_0^{t_0}\int_0^\infty \beta(a)p(t,a)\,da\,dt
\end{align*}
and
$$\int_0^{t_0} p(t,0)\,dt=2(1-f)\int_0^{t_0}\int_0^\infty \beta(a)p(t,a)\,da\,dt.$$
Finally we obtain
\beq\label{eq:frac2}\f{Q(t_0)}{Q(t_0)+\int_0^{t_0}p(t,0)\,dt}=\f{2f}{2f+2(1-f)}=f.\eeq
So when there is no death rate, the experimental fraction $F$ corresponds exactly to the coefficient $f.$
In our simulations we consider positive a death rate $\mu$ as suggested by the results in Section~\ref{ssec:imt_beta_mu}.
But because the numerical value we recover for $\mu$ is very small $(\mu\ll1$ in Figure~\ref{fig:imtfitmu}),
we can consider that $f$ is still well-approximated by $F$ and we use the fractions available in~\cite{TysonGarbett} to parameterize $f.$

The identification of the coefficients $\beta$ and $\mu$ is much more delicate because $\beta(a)$ is a distributed function.
The two following sections are devoted to presenting a method to recover these coefficients from IMT distributions.
Then these coefficients are used to compare the solutions of Equation~\eqref{eq:quiescence} to experimental data in Section~\ref{sec:numres}.

\

\section{Modelling the intermitotic time}\label{sec:imt}

Since the age of a cell is defined as the time elapsed from its last mitosis, the IMT of a cell is its age at division.
This definition allows us to interpret the IMT distributions in terms of a dynamic age-structured population model.

\

Experimental IMT distributions can be seen as histograms which represent, for a given population,
the density of cells with a certain age of division (see Figure~\ref{fig:imt} for an example taken from~\cite{TysonGarbett}).
The age of division is distributed into $N_a$ bins of width $\Delta a.$
For all $i$ between $1$ and $N_a,$ the height $H_i$ of the $i^{th}$ bar represents the density of cells with an age of division in the window $[i\Delta a,(i+1)\Delta a].$
The histogram $(H_i)_{1\leq i\leq N_a}$ is normalized to represent a density
\beq\label{eq:norm_hist}\Delta a\sum_{i=1}^{N_a} H_i=1.\eeq

\begin{figure}[h!]
\begin{minipage}{\linewidth}
\begin{center}
\includegraphics[width=.6\textwidth]{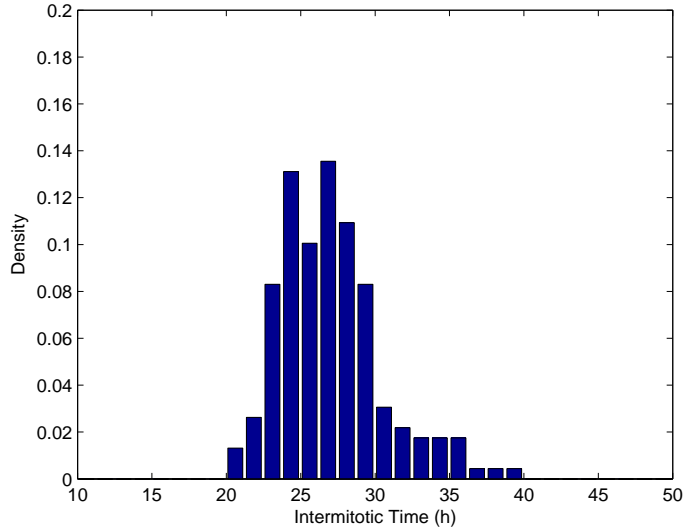}
\end{center}
\end{minipage}
\caption{Plot of an histogram $(H_i)_{1\leq i\leq N_a}$ representing the IMT distribution of a population of PC-9 lung cancer cells.
The data is from~\cite{TysonGarbett} and the bin width is $\Delta a=\f{10}{8}.$}\label{fig:imt}
\end{figure}

We briefly explain here the method used in~\cite{TysonGarbett} to build IMT histograms (we refer to this paper for more details).
The data are obtained using extended temporally resolved automated microscopy (ETRAM) in which cell nuclei are fluorescently labeled,
imaged by automated time lapse fluorescence microscopy, and tracked as individual cells from the resultant image stacks.
Cells are subjected to various microenvironmental conditions (such as the addition of a drug) at a specific time during image acquisition (set to time $t_0)$
and the effect of the perturbation from that point in time is followed across the entire population and individual cells within it.
The duration of observation $(T)$ is chosen large enough to observe that almost all cycling cells divide before this final time.
The data are organized into the bins to obtain a histogram, which is then normalized to represent a density of cells.

\

The next step is to describe the IMT distributions in terms of an age-structured model and use this interpretation to parameterize Equation~\eqref{eq:quiescence}.
We want to estimate the coefficients $\beta$ and $\mu,$ and for this we fix $f=0$ in~\eqref{eq:quiescence}.
This corresponds to the situation when there is no erlotinib treatment and thus all the cells are proliferating.
We recover in this case the original McKendrick--Von Foerster's model
\beq\label{eq:renewal}
\left\{\begin{array}{l}
\dis\f{\p}{\p t}p(t,a)+\f{\p}{\p a}p(t,a)+\beta(a)p(t,a)+\mu\, p(t,a)=0,\qquad t\geq0,\quad a>0,
\vspace{.3cm}\\
p(t,0)=\dis2\int_0^\infty \beta(a)p(t,a)\,da,
\vspace{.3cm}\\
p(0,a)=p_0(a).
\end{array}\right.
\eeq
It is a transport equation which satisfies a maximum principle, namely if the initial distribution $p_0(a)$ is nonnegative (positive)
then the distribution $p(t,a)$ remains nonnegative (positive) for all time $t>0.$
The solutions of~\eqref{eq:renewal} have a remarkable behavior as time evolves, in that the age structure equilibrates no matter what the initial age structure of the population may be.
This effect is known as \textit{asynchronous exponential growth}, and its interpretation is that the population disperses over time to a limiting asymptotic equilibrium age structure
where the fraction of the population in any age range $[a_1,a_2]$ satisfies  
$$\lim_{t \rightarrow \infty} \frac{\int_{a_1}^{a_2}p(t,a) \, da}{\int_{0}^{\infty}p(t,a) \, da}$$ = 
a constant independent of the initial age structure \cite{Webb1987}.
Moreover, the solutions to this equation, as $t \rightarrow \infty$, are known to behave like a separated variables solution, that is,
$p(t,a) =$ a function of time only $\times$ a function of age only.
More precisely, consider the eigenvalue problem
\beq\label{eq:eigen}\left\{\begin{array}{l}
\lambda \hat p(a) + \partial_a\hat p(a) + \beta(a)\hat p(a) + \mu\,\hat p(a)= 0,
\vspace{2mm}\\
\hat p(0) = 2\int_0^\infty \beta(a)\hat p(a)\,da,
\vspace{2mm}\\
\hat p(\cdot)>0,\quad \int \hat p(a)\,da=1,
\end{array}\right.
\eeq
This problem has a unique solution given by
$$\hat p(a)=\hat p(0)e^{-\int_0^a (\beta(a')+\mu+\lambda)\,da'}$$
where $\lambda>0$ is the unique value such that
\beq\label{eq:lambda}1=2\int_0^\infty\beta(a)e^{-\int_0^a(\beta(a')+\mu +\lambda)\,da'}da\eeq
and
$$\hat p(0)=\left(\int_0^\infty e^{-\int_0^a (\beta(a')+\mu+\lambda)\,da'}\,da\right)^{-1}.$$
Then we can prove that, for large times,
$$p(t,a)\sim const\,\hat p(a)e^{\lambda t}$$
(see Appendix~\ref{ap:equilibrium} for more details and references).
If the population of cells proliferates over a sufficiently long time, we can assume that this asymptotic behavior is reached and use it to investigate the IMT distributions.
An experimental  observation that the total population (independent of age structure) is growing exponentially is an indicator that the population has effectively reached the equilibrium age distribution,
which can be checked using ETRAM for other time series data collection (see Figure~\ref{fig:lambda}).

\begin{figure}[h]
\psfrag{Ln(N(t))}[l]{$\ln\bigl(\f{N(t)}{N(0)}\bigr)$}
\psfrag{t}[c]{$t\quad(in\ hour)$}
\begin{center}
\includegraphics[width=.6\textwidth]{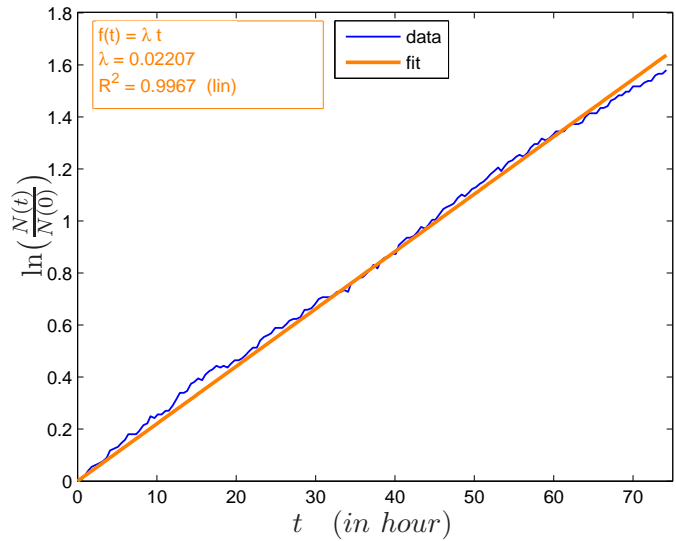}
\end{center}
\caption{The evolution of the total population $N(t)$ is plotted on $\log$-scale (blue) (the data is from~\cite{TysonGarbett}).
The data are well-fitted with a line of slope $\lambda = 0.022$ (correlation coefficient: $R^2=0.9967)$.
This is evidence that the total population is growing exponentially fast with exponential constant $\lambda$.
Thus, $N(t)=\exp^{\lambda t} N(0)$, and if $t^{\star}$ is the population doubling time, then $N(t^{\star}) = 2 N(0)$, and $t^{\star} = \ln 2 / \lambda$.}\label{fig:lambda}
\end{figure}

\

Now we give a continuous expression of the IMT distribution in terms of the age-structured model.
The age distribution of the cells relative to time $t_0$ (time of perturbation) is given by a truncation of the equilibrium age distribution $\hat p$
$$\bar p_0(a)=\left\{\begin{array}{ll}
               \rho\,e^{-\int_0^a(\beta(a')+\mu+\lambda)\,da'}&\text{if}\ 0\leq a\leq t_0,\\
               0&\text{if}\ a>t_0.
              \end{array}\right.$$
where $\rho$ is a scaling constant.
We then follow this age distribution along time and obtain for $t>0$
$$\bar p(t,a)=\left\{\begin{array}{ll}
               \rho\,e^{-\int_0^a(\beta(a')+\mu+\lambda)\,da'}e^{\lambda t}&\text{if}\ t\leq a\leq t+t_0,\\
               0&\text{otherwise.}
              \end{array}\right.$$
According to the age-structured model, and because of the normalization~\eqref{eq:norm_hist}, the IMT distribution satisfies
\beq\label{eq:IT}I_T(a):=C_T^{-1}\int_0^T \beta(a)\bar p(t,a)\,dt,\qquad C_T:=\int_0^\infty\int_0^T \beta(a)\bar p(t,a)\,dt\,da,\eeq
by definition of the division rate $\beta.$
The fact that no cell can divide in a time less than $t_0$ means
\beq\label{as:betat0}\forall\,a\leq t_0,\qquad\beta(a)=0.\eeq
Under this condition, for $T$ large, the function $I_T$ is close to the function
\beq\label{eq:Iinfty}I_\infty(a):= C_\infty^{-1}\beta(a)e^{-\int_0^a(\beta(a')+\mu)\,da'}\eeq
where
\beq\label{eq:Cinfty}C_\infty:=\int_0^\infty\beta(a)e^{-\int_0^a(\beta(a')+\mu)\,da'}\,da.\eeq
This convergence is made mathematically precise by the following claim
(the proof is given in Appendix~\ref{ap:claim} for a more general case in which it is not assumed that $p(t,a)$ is close to the equilibrium age distribution):

\begin{claim}\label{claim:conv}
Under Assumptions~\eqref{as:beta} and \eqref{as:betat0}, we have the convergence
$$\int_0^\infty|I_T(a)-I_\infty(a)|\,da\xrightarrow[T\to\infty]{}0.$$
\end{claim}

\begin{remark}
We can easily prove that, under the additional assumption
\beq\label{as:beta1}\lim_{a\to+\infty}\beta(a)e^{-\int_0^a\beta(a')\,da'}=0,\eeq
we also have
$$\sup_{a\geq0}|I_T(a)-I_\infty(a)|\xrightarrow[T\to+\infty]{}0.$$
Condition~\eqref{as:beta1} is satisfied from Assumption~\eqref{as:beta}, if for example, $\beta$ bounded or monotonic.
\end{remark}

Because in the experimental protocol $T$ is chosen large enough to observe no dividing cells at the end,
we can approximate $I_T$ by $I_\infty$ which has a simple expression in terms of the division and death rates~\eqref{eq:Iinfty}.

\section{From the intermitotic time to the division rate}

In this section we explain how we can recover the parameters $\beta$ and $\mu$ from the IMT distribution $I_\infty.$
We first present the method in the case when there is no death $(\mu=0).$
This simplification is useful to give an inversion formula which gives the rate $\beta$ in terms of $I_\infty.$
Then we extend the method to include possible death rates.

\subsection{The case $\mu=0$}\label{ssec:imt_beta}

In the case $\mu=0,$ the constant $C_\infty$ is equal to 1 and we have 
\[I_\infty(a)= \beta(a)e^{-\int_0^a\beta(a')\,da'}.\]
This expression can be inverted to recover the division rate from the IMT distribution (see \cite{HinowWebb,BiCFerGLOuS})
\beq\label{eq:inverse}\beta(a)=\f{I_\infty(a)}{\int_a^\infty I_\infty(a')\,da'}.\eeq
We start from the fitting of the experimental IMT distributions, which are observed to be positively skewed.
From a fitting procedure for  $I_\infty,$ we use Equation~\eqref{eq:inverse} to recover the division rate of the McKendrick--Von Foerster's equation.

\

First consider as in~\cite{HinowWebb} that the IMT distribution is a shifted gamma function (see also~\cite{Bernard} and references therein).
Setting
\beq\label{eq:gamma}I_\infty(a|m,\sigma)=\left\{\begin{array}{ll}
0&\text{if}\ 0\leq a\leq m,\\
\dis\f{a-m}{\sigma^2}e^{-\f{a-m}{\sigma}}&\text{if}\ a\geq m,
\end{array}\right.
\eeq
we can solve explicitly Equation~\eqref{eq:inverse} and we find
\beq\label{eq:beta_gamma}\beta(a)=\left\{\begin{array}{ll}
0&\text{if}\ 0\leq a\leq m,\\
\dis\f{a-m}{\sigma(\sigma+a-m)}&\text{if}\ a\geq m.
\end{array}\right.
\eeq
So by fitting an experimental IMT distribution with a shifted gamma function, we obtain two parameters $m$ and $\sigma$ which allows reconstruction of the division rate of the renewal equation.

To have a smoother transition at the minimum age of division $m,$ one can consider a second shifted gamma function
\beq\label{eq:gamma2}I_\infty(a|m,\sigma)=\left\{\begin{array}{ll}
0&\text{if}\ 0\leq a\leq m,\\
\dis\f{(a-m)^2}{2\sigma^3}e^{-\f{a-m}{\sigma}}&\text{if}\ a\geq m.
\end{array}\right.
\eeq
Then the corresponding $\beta$ is
\beq\label{eq:beta_gamma2}\beta(a)=\left\{\begin{array}{ll}
0&\text{if}\ 0\leq a\leq m,\\
\dis\f1\sigma\f{(a-m)^2}{2\sigma^2+2\sigma(a-m)+(a-m)^2}&\text{if}\ a\geq m.
\end{array}\right.
\eeq
The different functions~\eqref{eq:gamma} to~\eqref{eq:beta_gamma2} are plotted in Figure~\ref{fig:inverse_gamma}

\begin{figure}[h]
\begin{minipage}{.49\linewidth}
\begin{center}
\includegraphics[width=\textwidth]{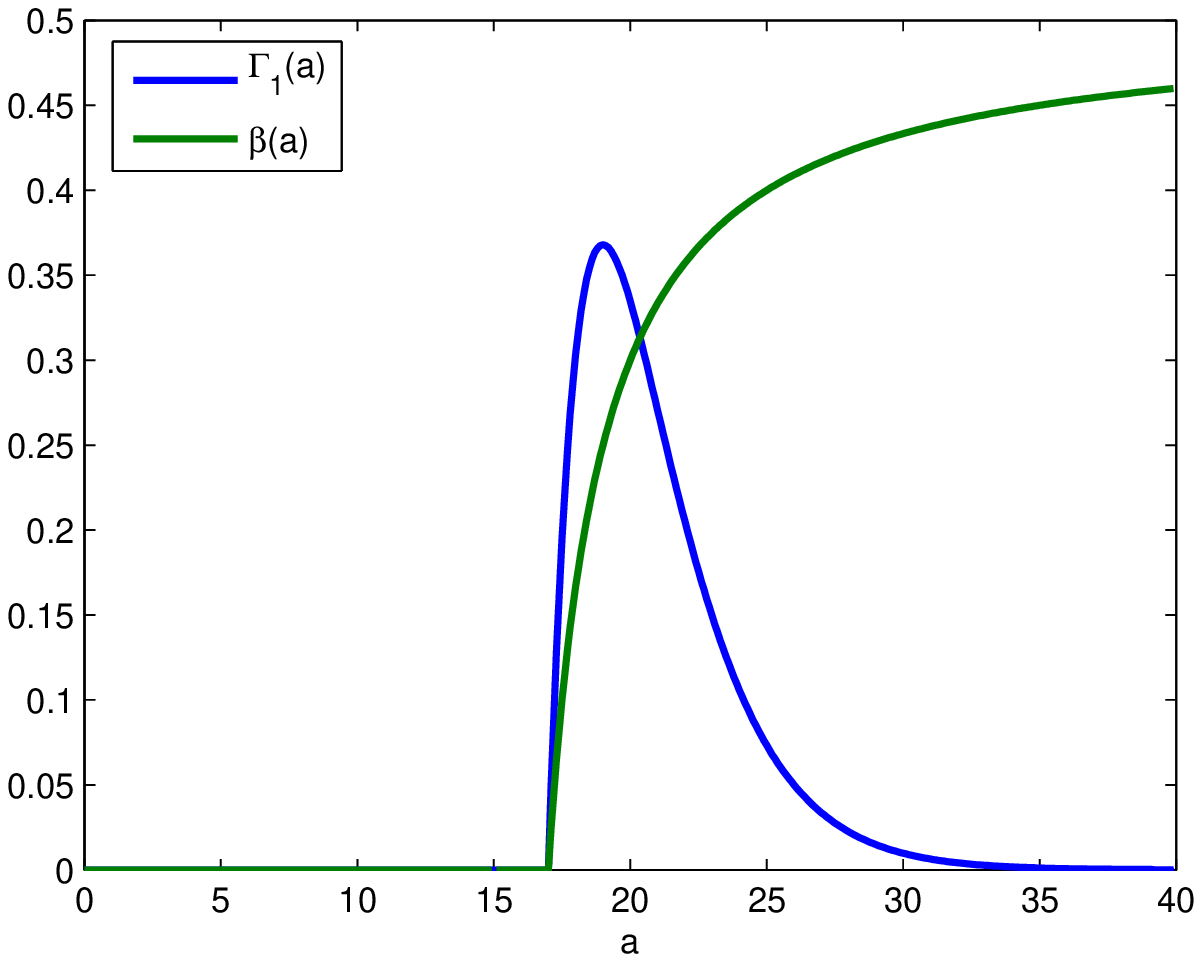}
\end{center}
\end{minipage}\hfill
\begin{minipage}{.49\linewidth}
\begin{center}
\includegraphics[width=\textwidth]{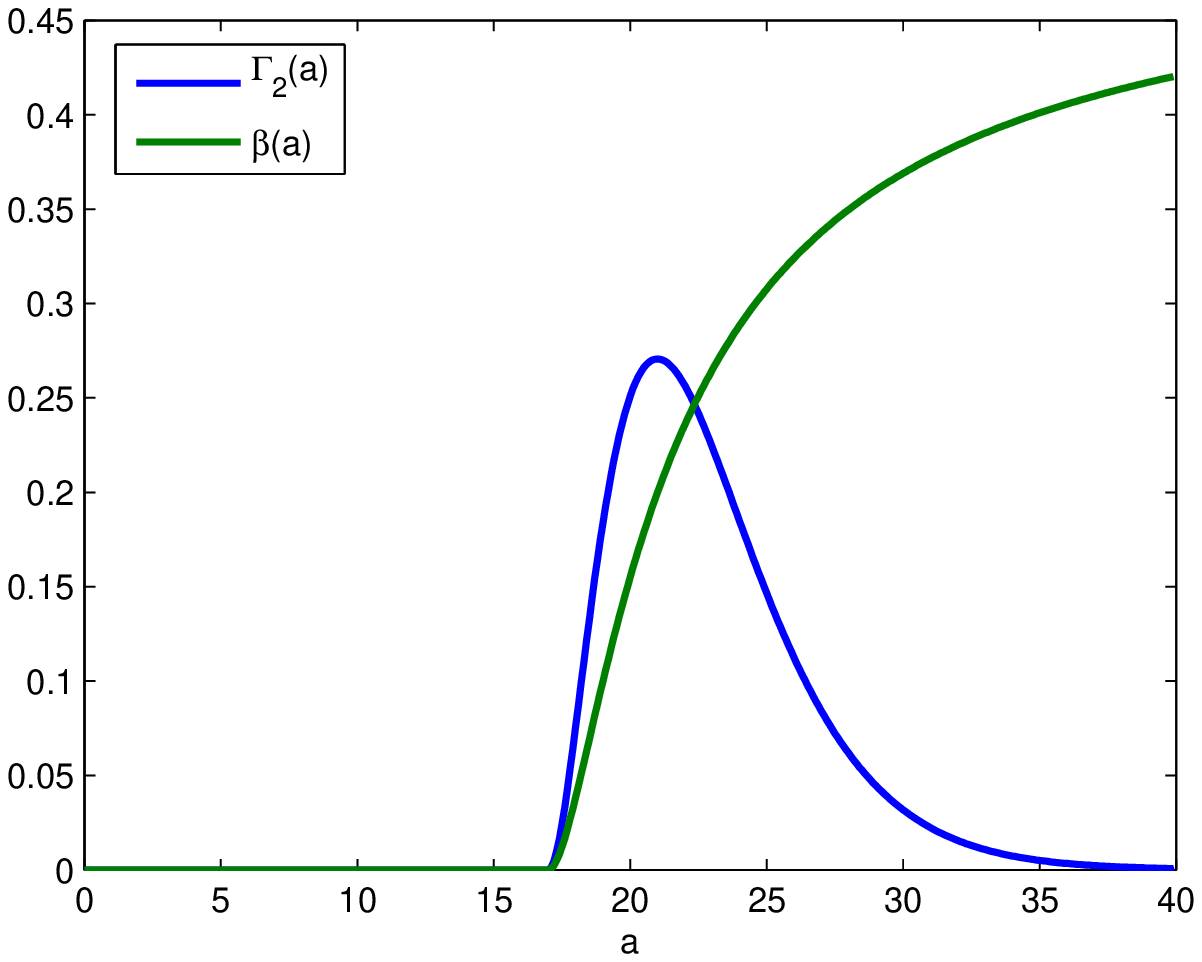}
\end{center}
\end{minipage}
\caption{The two different gamma functions and their corresponding division rate $\beta(a)$ are plotted for coefficients $m=17$ and $\sigma=2.$}\label{fig:inverse_gamma}
\end{figure}

\

It has been observed that an exponentially modified Gaussian (EMG) is often a better model for IMT distributions than the gamma function~\cite{Golubev,TysonGarbett} .
An EMG is defined as the convolution of a Gaussian with a decreasing exponential, but after solving it can be written with three parameters as
\beq\label{eq:emg}I_\infty(a|\beta_0,m,\sigma)=\beta_0\,Erfc\,\Bigl(\f{m-a}{\sigma}\Bigr) e^{-2\beta_0\bigl(\f{\beta_0\sigma^2}{2}-m+a\bigr)}.
\eeq
where the (complementary) error function is defined by 
$$Erfc(z) = 1-\frac{2}{\sqrt{\pi}} \, \int_0^z e^{-t^2} \,dt.$$
Replacing $I_\infty$ by an EMG in Equation~\eqref{eq:inverse}, we cannot compute explicitly the expression for $\beta.$ But by numerical comparison,
we obtain a division rate $\beta$ that is essentially indistinguishable from an error function (see Figure~\ref{fig:inverse}).

\begin{figure}[h!]
\begin{minipage}{.49\linewidth}
\begin{center}
\includegraphics[width=\textwidth]{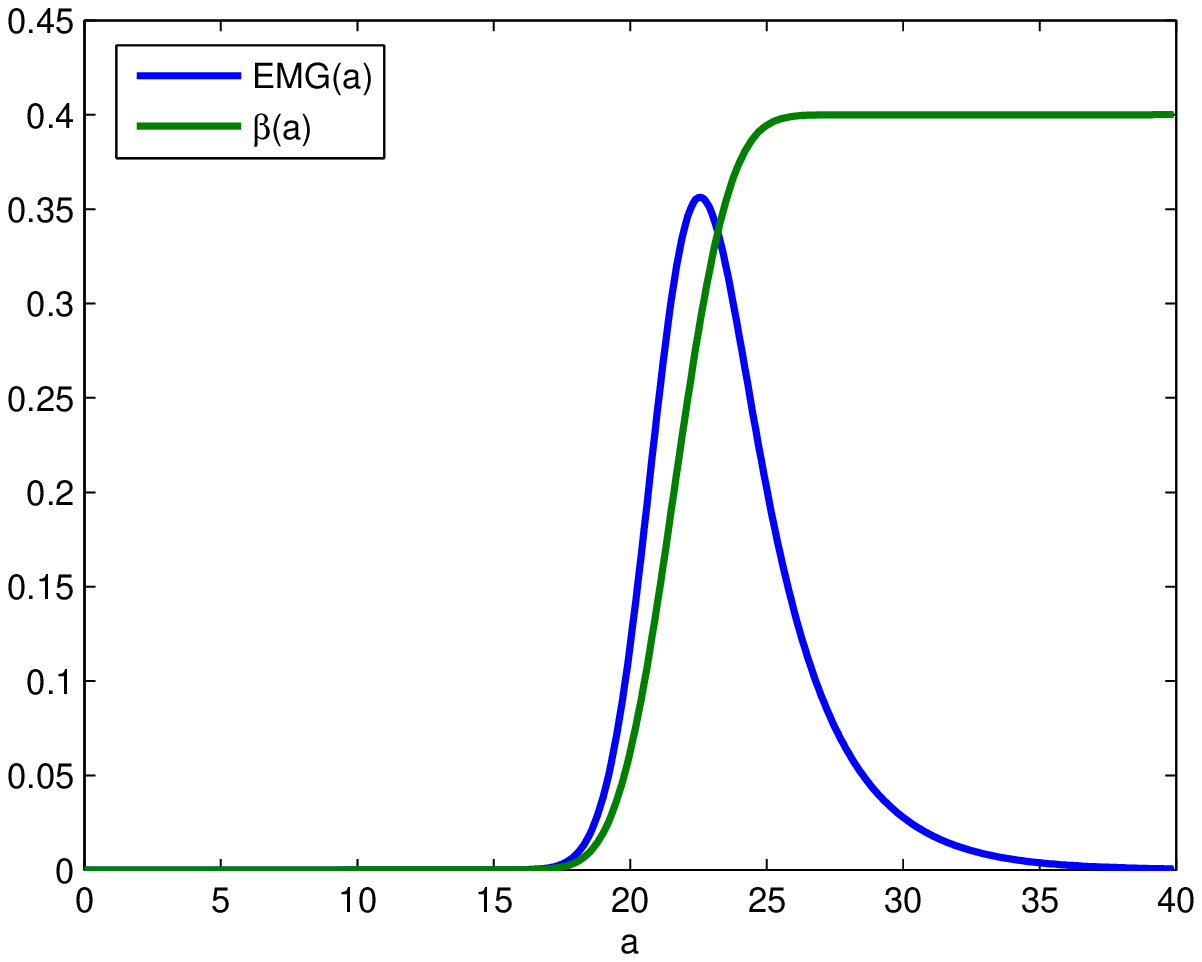}
\end{center}
\end{minipage}\hfill
\begin{minipage}{.49\linewidth}
\begin{center}
\includegraphics[width=\textwidth]{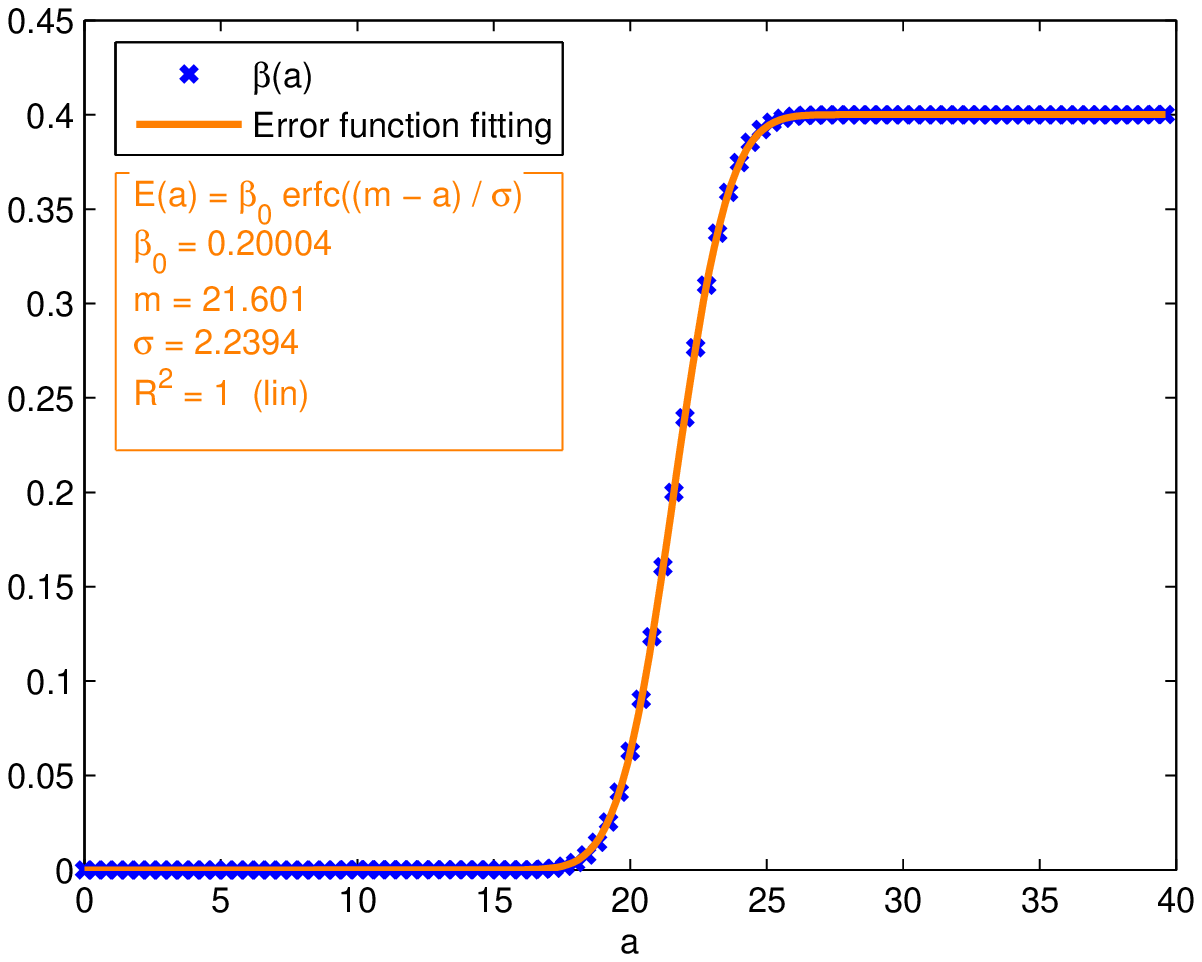}
\end{center}
\end{minipage}
\caption{Left: The division rate $\beta(a)$ is obtained numerically from~\eqref{eq:inverse} by using an EMG with coefficients $m=22,$ $\sigma=2$ and $\beta_0=0.2$ to fit  $I_\infty$~\eqref{eq:emg}.
Right: The formula for $\beta$, as an error function with these same three parameters, is numerically indistinguishable from the numerically obtained $\beta$ $(R^2=1)$.}\label{fig:inverse}
\end{figure}

Instead of fitting the IMT distribution by an EMG and then fitting the corresponding $\beta$ by an error function,
we may directly assume that $\beta$ is an error function
\beq\label{eq:beta_error}\beta(a)=\beta_0\, Erfc\Bigl(\f{m-a}{\sigma}\Bigr).\eeq
We can then explicitly derive a new fitting formula for $I_\infty$ due to Equation~\eqref{eq:Iinfty}
\beq\label{eq:imt_error}I_\infty(a|\beta_0,m,\sigma)=\beta_0\, Erfc\Bigl(\f{m-a}{\sigma}\Bigr)e^{-\int_0^a \beta_0\, Erfc(\f{m-a'}{\sigma})\,da'}\eeq
where the integral in the exponential can be computed as
\beq\label{eq:integral}\int_0^a Erfc\Bigl(\f{m-a'}{\sigma}\Bigr)\,da'=m\,Erfc\Bigl(\f{m}{\sigma}\Bigr)-\f{\sigma}{\sqrt\pi}e^{-(\f{m}{\sigma})^2}
-(m-a)Erfc\Bigl(\f{m-a}{\sigma}\Bigr)+\f{\sigma}{\sqrt\pi}e^{-(\f{m-a}{\sigma})^2}.\eeq
Using this formula for the IMT instead of an EMG formula, the fitting parameters provide immediately the division rate $\beta.$

\

When we know the Malthusian growth parameter (population growth rate $\lambda)$ from experimental data (see Figure~\ref{fig:lambda}),
it is possible to recover a division rate $\beta$ such that relation~\eqref{eq:lambda} is satisfied.
This is an important step in order to parameterize and apply an age-structured model to experimental data.
We notice the relationship of $\beta$ in terms of $I_\infty(a)$ is
$$2\int_0^\infty I_\infty(a)e^{-\lambda a}\,da=2\int_0^\infty\beta(a)e^{-\int_0^a\beta(a')\,da'}e^{-\lambda a}da=1.$$
To take advantage of this relationship between the IMT data, $\beta(a),$ and the malthusian parameter $\lambda$ at the equilibrium distribution,
we define a new histogram $(\tilde H_i)_{1\leq i\leq N}$ by
\beq\label{eq:Htilde}\forall\,i,\qquad\tilde H_i:=\f{2H_ie^{-\lambda a_i}}{\Delta a\sum_{i=1}^N 2H_ie^{-\lambda a_i}}\eeq
where $a_i:=(i+\f12)\Delta a$ is the mean age of the $i^{th}$ bar.
Thus defined, the new histogram incorporates information about $\lambda$ and satisfies the relation
\beq\label{eq:norm_hist_tilde}\Delta a\sum_{i=1}^\infty\tilde H_i=1.\eeq
We then fit this new histogram $(\tilde H_i)$ with the model
\beq\label{eq:tildeIinfty}\tilde I_\infty(a):=2I_\infty(a)e^{-\lambda a}\eeq
instead of fitting $(H_i)$ with $I_\infty(a).$
Because of the normalization~\eqref{eq:norm_hist_tilde}, we expect that the fitting provides parameters such that $\int_0^\infty \tilde I_\infty(a)\,da\approx 1$, and
this relation can be checked numerically {\it a posteriori}.

\

\noindent{\bf Method and example.}
We divide the process in three steps and illustrate it by an example. Each required fitting step can be performed using the freely available Ezyfit Matlab toolbox\\
$[<\mbox{\url{http://www.fast.u-psud.fr/ezyfit/}}>].$

\

\begin{itemize}
 \item[]\underline{Step 1: Determine equilibrium IMT distribution}

\begin{enumerate}
 \item[a)] Obtain a histogram $(H_i)$ for the experimental IMT distribution of control (untreated) cells.
{\it Here we use an IMT distribution obtained  in~\cite{TysonGarbett} for PC-9 lung cancer cells using  ETRAM (see Figure~\ref{fig:imt}).}
 \item[b)] Plot the time evolution of the total population on a log$_e$-scale,
verify a linear fit and obtain the slope as the experimental value for the Malthusian parameter $\lambda$ of the cell population {\it (see Figure~\ref{fig:lambda} for an example)}.
 \item[c)] Construct the new histogram $(\tilde H_i)$ from $(H_i)$ and $\lambda$ by using the definition~\eqref{eq:Htilde}.
\end{enumerate}

 \item[]\underline{Step 2: Obtain parameters from model fit to IMT distribution}
\begin{enumerate}
 \item[a)] Choose a form for $I_\infty$ as a gamma function~\eqref{eq:gamma} or~\eqref{eq:gamma2}, or as the new EMG form~\eqref{eq:imt_error}.
 \item[b)] Fit the histogram $(\tilde H_i)$ with the corresponding form $\tilde I_\infty$ from definition~\eqref{eq:tildeIinfty}.
{\it For PC-9 cancer cells we choose the form~\eqref{eq:imt_error}, because in~\cite{TysonGarbett} it was observed that the IMT distribution appeared to be an EMG}
{\it (see Figure~\ref{fig:imtfit} for the example)}.
 \item[c)] Verify that the correlation coefficient $R^2$ of the IMT data and the chosen form  $\tilde I_\infty(a)$ is close to $1.$
\end{enumerate}

 \item[]\underline{Step 3: Parameterize age-structured model}

If numerical integral $\int_0^\infty \tilde I_\infty(a)\,da$ is close to $1,$
the fitting parameters provide a good approximation of the division rate $\beta,$ which can then be propagated through the population, based on the choice of $I_\infty.$
{\it In the example, $\beta$ is given by Equation~\eqref{eq:beta_error} with the numerical paramaters of Figure~\ref{fig:imtfit}.}

\end{itemize}

\begin{figure}[h]
\begin{minipage}{\linewidth}
\begin{center}
\includegraphics[width=.6\textwidth]{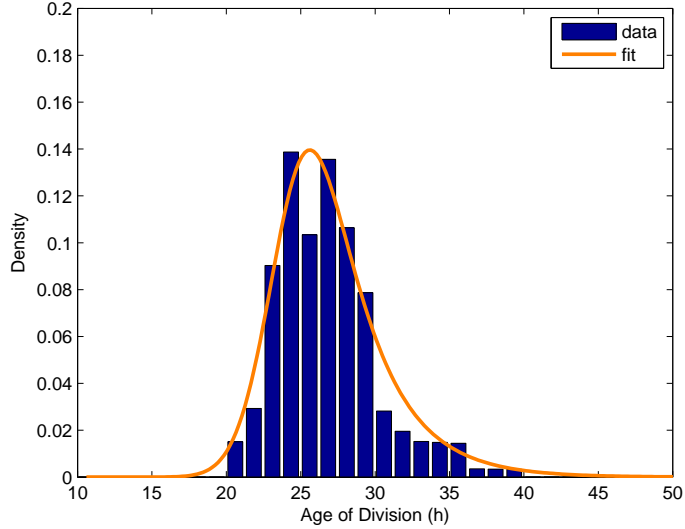}
\end{center}
\end{minipage}
\caption{Fitting of experimental data in~\cite{TysonGarbett} with the model~\eqref{eq:imt_error}. The fitting parameters are $\beta_0=0.14204,$ $m=24.456$ and $\sigma=3.3451.$
The correlation coefficient is $R^2=0.95363$ and the integral of $\tilde I_\infty(a|\beta_0,m,\sigma)$ is $\int_0^\infty \tilde I_\infty(a)\,da\approx 1.0983.$ The formula for the age dependent division rate is $\beta(a) = \beta_0 Erfc(\frac{m-a}{\sigma})$.}
\label{fig:imtfit}
\end{figure}

\

\subsection{Introducing a death rate}\label{ssec:imt_beta_mu}

If we include a death rate $\mu>0$ in the model, then there does not exist an inversion formula as Equation~\eqref{eq:inverse} to recover $\beta$ analytically.
So we consider that $\beta$ is a rational fraction given by Equation~\eqref{eq:beta_gamma} or~\eqref{eq:beta_gamma2}, or an error function given by Equation~\eqref{eq:beta_error},
and we derive the corresponding model $I_\infty$ from Equation~\eqref{eq:Iinfty}.
The death rate $\mu$ will appear in the expression of $I_\infty$ as an additional fitting parameter
and will be determined together with the parameters of $\beta$ by the fitting of experimental IMT distributions.
Notice that a fixed rate $\mu$ based on measured data can also be used if available to avoid the addition of a free parameter.

\

Unlike the case $\mu=0$ for which $C_\infty=1,$ we have for a positive death rate that $C_\infty<1.$
This constant is a function of the fitting parameters (see Equation~\eqref{eq:Cinfty}), but we do not have an analytic expression of this function in general.
A solution to this problem is to use the Malthusian parameter $\lambda,$ which incorporates information about $\mu$ and $\beta(a),$
to define a form $\tilde I_\infty(a)$ as in Section~\ref{ssec:imt_beta}.
Indeed we obtain, because of relation~\eqref{eq:lambda}, 
$$\tilde I_\infty(a):=2\beta(a)e^{-\int_0^a(\beta(a')+\mu)\,da'}e^{-\lambda a}$$
which does not involve $C_\infty.$
Then, the division rate $\beta(a)$ is obtained from $\tilde I_\infty(a)$ as before using the modified histogram $(\tilde H_i)$ defined from $(H_i)$ by~\eqref{eq:Htilde}.

\

\noindent\underline{Example.}
We fit the same distribution as in Section~\ref{ssec:imt_beta} still considering that $\beta$ is an error function.
With a constant death rate $\mu,$ we obtain the four parameters model
\beq\label{eq:imt_error_mu}\tilde I_\infty(a|\beta_0,m,\sigma,\mu)=2\beta_0\, Erfc\,\Bigl(\f{m-a}{\sigma}\Bigr)\,e^{-\int_0^a (\beta_0\, Erfc\,(\f{m-a'}{\sigma})+\mu +\lambda)\,da'}\eeq
where the integral $\int_0^a\beta_0\, Erfc\bigl(\f{m-a'}{\sigma}\bigr)\,da'$ is given by Equation~\eqref{eq:integral}.
The fitting provides new parameters for $\beta$ and a positive death rate $\mu$ (see Figure~\ref{fig:imtfitmu}).
The correlation $R^2$ is slightly better than in Figure~\ref{fig:imtfit} and the integral $\int_0^\infty\tilde I_\infty(a)\,da$ is significantly closer to 1.
So we can assume that mortality has to be considered for this cell line.

\begin{figure}[h]
\begin{minipage}{\linewidth}
\begin{center}
\includegraphics[width=.6\textwidth]{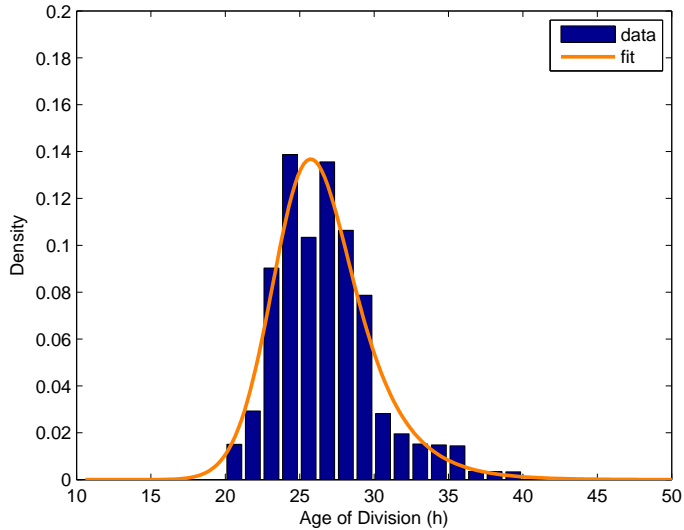}
\end{center}
\end{minipage}
\caption{Fitting of experimental data with the model~\eqref{eq:imt_error_mu}. The fitting parameters are $\beta_0=0.17879,$ $m=25.007,$ $\sigma=3.6141$ and $\mu=0.00333.$
The correlation coefficient is $R^2=0.95611$ and the integral of $\tilde I_\infty(a|\beta_0,m,\sigma)$ is $\int_0^\infty\tilde I_\infty(a)\,da\approx 1.0132.$
The formula for the age dependent division rate is $\beta(a) = \beta_0 Erfc(\frac{m-a}{\sigma})$.}\label{fig:imtfitmu}
\end{figure}

\

\section{Numerical simulation and discussion}\label{sec:numres}

Once the parameters accurately estimated,
we use them to compare the behavior of the solution of Equation~\eqref{eq:quiescence} to the experimental data in~\cite{TysonGarbett}.
Since the model does not have analytic solutions, we need to perform numerical simulations.
Many numerical methods are available for structured population equations (see for instance~\cite{Ackleh2,Angulo1,Angulo2,DouglasMilner,GabrielTine,Kostova}).
Here we use a scheme based on the method of characteristics, as in~\cite{Angulo1,Angulo2,Kostova}, for its anti-dissipative properties.

\

We can see in Figure~\ref{fig:compare} that the numerical simulations are very similar to the experimental curves.
In particular, the delay of twenty hours before the effect of the treatment on the growth of the population is apparent,
and this twenty hour period pulses two more times as the population approaches a new equilibrium distribution during the total time of the experiment.
We have thus developed an age-structured model that can explain the dynamic effects of erlotinib on PC-9 cells,
which are intrinsically dependent on the age of proliferating cells.
The model is also accurate in a quantitative point of view.
The quiescent fractions $f$ which are used in the numerical simulations are not chosen arbitrarily to obtain the adequate behavior.
They are linked to the treatment dose of erlotinib and estimated from experimental data in a rigorous way so that they are realistic.

The model has been chosen as simple as possible to be able to recover all the parameters from experimental IMT distributions.
This simplicity leads to qualitative differences between observed data and simulations in Figure~\ref{fig:compare}.
Experimental observation suggests the population reaches equilibrium when treated with drug whereas simulated population size is still increasing at the end of the experiment.
An age-structure for quiescent cells together with an age-dependent rate of death should be consider to explain this plateau effect.
Experimental data also suggest that drug treatment increases population growth rate above that of untreated cells briefly, but model does not capture this.
This could be obtained by considering a lower death rate for quiescent cells than for proliferating cells.
But for such more complex models, additional experimental data and a new parameter estimation method would be necessary to estimate the death rates.

\begin{figure}[h]
\psfrag{Ln(N(t))}[l]{$\ln\bigl(\f{N(t)}{N(0)}\bigr)$}
\psfrag{Ln(P(t)+Q(t))}[l]{\small$\ln(P(t)+Q(t))$}
\psfrag{t (h)}[c]{$t\ (h)$}
\begin{center}
\includegraphics[width=.9\textwidth]{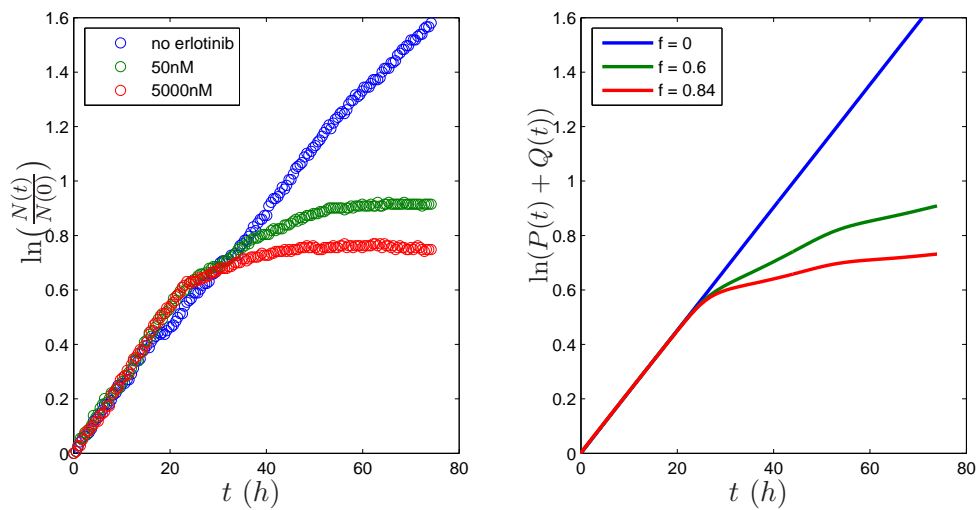}
\end{center}
\caption{Left: experimental data without erlotinib, with a $50nM$ dose and with a $5000nM$ dose.
The total quantity $N(t)$ is plotted on a log-scale.
Right: numerical simulation of model~\eqref{eq:quiescence} for $f=0,$ $f=0.6$ and $f=0.84.$
The curves represent the evolution of $\ln(P(t)+Q(t))$ where $P(t):=\int_0^\infty p(t,a)\,da$ is the total quantity of proliferating cells at time $t.$
They are obtained by solving Equation~\eqref{eq:quiescence} numerically with the parameters of Figure~\ref{fig:imtfitmu} and with $\nu=0.004.$}\label{fig:compare}
\end{figure}

\

\section*{Conclusion}

Linking experimental observation of cell behavior between the single-cell and population scales has recently been described using newly developed mathematical models~\cite{TysonGarbett}.
However, this approach does not take into account the possible cell age-dependent effects of a perturbation on cell behavior,
such as would be expected if the effects occur in cells at a specific position in the cell cycle.
Since these studies were performed with asynchronously dividing cell populations, it is evident that the mathematical models of these experiments should have age structure as a primary feature.
In fact, to fit the data, the authors had to use an artificial time offset to account for the age-structured effects.
Here we provide a formalized approach to accurately account for cell age-dependent effects on cellular behavior.
A major difficulty in the parameterization of age-structured models is the determination of the age dependent division rate.
Our study provides a method for the quantitative recovery of this rate by fitting experimental IMT distributions to special forms, such as gamma functions or exponential modified Gaussians.
This model, once successfully parameterized, is very useful for simulating and analyzing age-dependent phenomena in cell population dynamics. 
We have presented one such application for the \textit{in vitro} treatment of cancer cells by erlotinib.
This example shows the utility of age-structured population models in explaining the connection of drug therapy to  phenomena such as cell cycle phase entry into quiescence.
The method we have presented can be implemented readily to many issues in cell population behavior when there is experimental data based on cell age.
\

\section*{Aknowledgments}

The research stays of P. Gabriel at Vanderbilt University have been financially supported by a grant of the \emph{Fondation Pierre Ledoux Jeunesse Internationale}.

\

\appendix

\section{Convergence to the equilibrium}\label{ap:equilibrium}

The long-time behavior $p(t,a)\sim const\,\hat p(a)e^{\lambda t}$ can be proven by using semi-groups methods~\cite{MetzDiekmann,Webb} or General Relative Entropy techniques~\cite{MMP2,BP}.
We detail here a result provided by General Relative Entropy.

Consider $\phi$ the unique solution to the adjoint eigenvalue problem
$$\left\{\begin{array}{l}
\lambda \phi(a) - \partial_a\phi(a) + \beta(a)\phi(a) + \mu\,\phi(a) = 2\phi(0)\beta(a),
\vspace{2mm}\\
\phi(\cdot)\geq0,\quad \int\hat p(a)\phi(a)\,da=1.
\end{array}\right.
$$
Then the General Relative Entropy method allows us to prove that
$$\int_0^\infty |p(t,a)e^{-\lambda t}-\rho_0\hat p(a)|\phi(a)\,da\xrightarrow[t\to\infty]{}0$$
where
$$\rho_0=\int_0^\infty\phi(a)p_0(a)\,da.$$

\section{Convergence of the IMT distribution}\label{ap:claim}

\begin{theorem}\label{th:ITconv}
Suppose that Assumptions~\eqref{as:beta} and \eqref{as:betat0} are satisfied and consider an initial distribution $\bar p_0$ such that $\bar p_0(a)=0$ for all $a>t_0.$
Then we have the convergence
$$\int_0^\infty|I_T(a)-I_\infty(a)|\,da\xrightarrow[T\to\infty]{}0$$
where $I_T$ and $I_\infty$ are defined in~\eqref{eq:IT} and \eqref{eq:Iinfty}.
\end{theorem}

For the sake of simplicity, the proof of Theorem~\ref{th:ITconv} is given in the case $\mu=0.$
But it can be easily adapted to the case with a death rate.

\begin{proof}[Proof of Theorem~\ref{th:ITconv} in the case $\mu=0$]

First we extend the initial distribution and the division rate $\beta(a)$ to negative ages by setting $\bar p_0(a)=\beta(a)=0$ for $a<0.$
Since the daughters of the labeled cells are not tracked, the age distribution $\bar p(t,a)$ of labeled cells satisfies a transport equation without boundary condition.
Thus it writes, using the characteristic method,
$$\forall\,t,a\geq0,\qquad \bar p(t,a)=\bar p_0(a-t)e^{-\int_0^t\beta(a-s)\,ds}.$$
Introducing this expression in the definition of $I_T(a)$ we obtain, with changes of variables,
\begin{align*}
I_T(a)&=C_T^{-1}\int_0^T\beta(a)\bar p_0(a-t)e^{-\int_0^t\beta(a-s)\,ds}dt\\
&=C_T^{-1}\int_0^T\beta(a)\bar p_0(a-t)e^{-\int_{a-t}^a\beta(a')\,da'}dt\\
&=C_T^{-1}\int_{a-T}^a\beta(a)\bar p_0(u)e^{-\int_{u}^a\beta(a')\,da'}du\\
&=C_T^{-1}\beta(a)e^{-\int_0^a\beta(a')\,da'}\int_{a-T}^a\bar p_0(u)e^{\int_0^u\beta(a')\,da'}du.
\end{align*}
Since the support of $\bar p_0$ is included in $[0,t_0]$ and $\beta(a)=0$ for $a\in[0,t_0]$ due to Assumption~\eqref{as:beta}, we have for all $u\in\R$
$$\bar p_0(u)e^{\int_0^u\beta(a')\,da'}=\bar p_0(u)$$
so $I_T(a)$ writes
$$I_T(a)=C_T^{-1}\beta(a)e^{-\int_0^a\beta(a')\,da'}\int_{a-T}^a\bar p_0(u)du.$$
Define the primitive
$$P(a):=\int_0^a\bar p_0(u)\,du.$$
This function is nondecreasing and satifies $P(a)=0$ for $a\leq0$ and $P(a)=P(t_0)=\int_0^\infty\bar p_0$ for $a\geq t_0,$ so we have
$$\int_{a-T}^a\bar p_0(u)du=P(a)-P(a-T)=
\left\{\begin{array}{ll}
P(a)&\text{if}\ 0\leq a\leq t_0,\\
P(t_0)&\text{if}\ t_0\leq a\leq T,\\
P(t_0)-P(a-T)&\text{if}\ T\leq a\leq T+t_0,\\
0&\text{if}\ a\geq T+t_0.
\end{array}\right.
$$
Because $\beta(a)=0$ for $a\leq t_0,$ we obtain by integration of $I_T(a)$ on $\R^+$ 
$$1=C_T^{-1}P(t_0)\int_0^T\beta(a)e^{-\int_0^a\beta(a')\,da'}\,da+C_T^{-1}\int_T^{T+t_0}\beta(a)e^{-\int_0^a \beta(a')\,da'} (P(t_0)-P(a-T))\,da$$
which gives
\begin{align*}
C_T-P(t_0)\int_0^T\beta(a)e^{-\int_0^a\beta(a')\,da'}\,da&=\int_T^{T+t_0}\beta(a)e^{-\int_0^a \beta(a')\,da'} (P(t_0)-P(a-T))\,da\\
&\leq P(t_0)\int_T^{T+t_0}\beta(a)e^{-\int_0^a \beta(a')\,da'}\,da\\
&\xrightarrow[T\to+\infty]{}0.
\end{align*}
Since
$$\int_0^\infty\beta(a)e^{-\int_0^a\beta(a')\,da'}\,da=1,$$
we conclude that
$$C_T\xrightarrow[T\to+\infty]{}P(t_0)=\int_0^\infty\bar p_0(a)\,da.$$
Then we have
\begin{align*}
\int_0^\infty\left|I_T(a)-I_\infty(a)\right|\,da&\leq\left|C_T^{-1}P(t_0)-1\right|\int_0^T\beta(a)e^{-\int_0^a\beta(a')\,da'}\,da\\
&\hspace{1cm}+\int_T^{T+t_0}\left|C_T^{-1}[P(t_0)-P(a-T)]-1\right|\beta(a)e^{-\int_0^a\beta(a')\,da'}\,da\\
&\hspace{2cm}+\int_{T+t_0}^\infty \beta(a)e^{-\int_0^a\beta(a')\,da'}\,da\\
&\leq \left|C_T^{-1}P(t_0)-1\right|\int_0^T\beta(a)e^{-\int_0^a\beta(a')\,da'}\,da + K \int_T^\infty \beta(a)e^{-\int_0^a\beta(a')\,da'}\,da \\
&\xrightarrow[T\to\infty]{}0
\end{align*}
and it ends the proof of Theorem~\ref{th:ITconv}.

\end{proof}

%
%

\bibliographystyle{abbrv}
\bibliography{erlotinib}

\end{document}